\documentclass{amsproc}
\usepackage{amssymb}
\usepackage{amsfonts}

\theoremstyle{plain}

\newtheorem{lemma}{Lemma}

\newtheorem{theorem}{Theorem}
\numberwithin{equation}{section}
\input{tcilatex}

\begin{document}
\title[A generalization of Combinatorial Nullstellensatz]{A generalization
of Combinatorial Nullstellensatz}
\author{Micha\l\ Laso\'{n}}
\address{Institute of Mathematics of the Polish Academy of Sciences, \'{S}niadeckich 8, 00-956 Warszawa, Poland}
\address{Theoretical Computer Science Department, Faculty of Mathematics and
Computer Science, Jagiellonian University, \L ojasiewicza 6, 30-348 Krak\'{o}w, Poland}
\keywords{Combinatorial Nullstellensatz}

\begin{abstract}
In this note we give an extended version of Combinatorial Nullstellensatz,
with weaker assumption on the degree of nonvanishing monomial. We also
present an application of our result in a situation where the original
theorem does not seem to work.
\end{abstract}

\maketitle

\section{Introduction}

The following theorem of Alon, known as Combinatorial Nullstellensatz, has
numerous applications in Combinatorics, Graph Theory, and Additive Number
Theory (see \cite{AlonCPC}).

\begin{theorem}
\label{CN}\emph{(Combinatorial Nullstellensatz \cite{AlonCPC})} Let $\mathbb{%
F}$ be an arbitrary field, and let $f$ be a polynomial in $\mathbb{F}%
[x_{1},...,x_{n}]$. Suppose the coefficient of $x^{\alpha _{1}}\cdots
x_{n}^{\alpha _{n}}$ in $f$ is nonzero and $\deg (f)=\sum_{i=1}^{n}\alpha
_{i}$. Then for any subsets $A_{1},\ldots ,A_{n}$ of $\mathbb{F}$ satisfying 
$\left\vert A_{i}\right\vert \geq \alpha _{i}+1$, there are $a_{1}\in
A_{1},\ldots ,a_{n}\in A_{n}$ so that $f(a_{1},\ldots ,a_{n})\neq 0$.
\end{theorem}

In this paper we extend this theorem by weakening the assumption on the
degree of nonvanishing monomial. We also provide an explicit formula for
coefficients of monomials in the usual expansion of $f$. Similar results
were obtained independently by Schauz \cite{Schauz}, however our proofs are
simple and more direct. The paper is concluded with an application to a
graph labeling problem for which classical approach does not seem to work.

\section{Generalized Combinatorial Nullstellensatz}

Let $\mathbb{F}$ be an arbitrary field, and let $f$ be a polynomial in $%
\mathbb{F}[x_{1},...,x_{n}]$. We define the \emph{support} of $f$ by $%
\limfunc{Supp}(f):=\{(\alpha _{1},\ldots ,\alpha _{n})\in \mathbb{N}^{n}:\;$%
the coefficient of $x_{1}^{\alpha _{1}}\cdots x_{n}^{\alpha _{n}}$ in $f$ is
nonzero$\}$. On the set $\mathbb{N}^{n}$ and hence also on $\limfunc{Supp}(f)
$ we have natural partial order: $(\alpha _{1},\ldots ,\alpha _{n})\geq
(\beta _{1},\ldots ,\beta _{n})$ if and only if $\alpha _{i}\geq \beta _{i}$
for all $i$. The proof of the following theorem is a simple extension of an
argument found by Micha\l ek \cite{Michalek}.

\begin{theorem}
\label{GCN}\emph{(Generalized Combinatorial Nullstellensatz)} Let $\mathbb{F}
$ be an arbitrary field, and let $f$ be a polynomial in $\mathbb{F}%
[x_{1},\ldots ,x_{n}]$. Suppose $x_{1}^{\alpha _{1}}\cdots x_{n}^{\alpha
_{n}}$ is a nonvanishing monomial in $f$ and $(\alpha _{1},\ldots ,\alpha
_{n})$ is maximal in $\limfunc{Supp}(f)$. Then for any subsets $A_{1},\ldots
,A_{n}$ of $\mathbb{F}$ satisfying $\left\vert A_{i}\right\vert \geq \alpha
_{i}+1$, there are $a_{1}\in A_{1},\ldots ,a_{n}\in A_{n}$ so that $%
f(a_{1},\ldots ,a_{n})\neq 0$.
\end{theorem}

\begin{proof}
The proof is by induction on $\alpha _{1}+\ldots +\alpha _{n}$. If $\alpha
_{1}+\ldots +\alpha _{n}=0$, then $f\equiv c\neq 0$ and the assertion is
true. If $\alpha _{1}+\ldots +\alpha _{n}>0$, then, without loss of
generality, we can assume that $\alpha _{1}>0$. Fix $a\in A_{1}$ and divide $%
f$ by $(x_{1}-a)$. So, we have 
\begin{equation*}
f=g\cdot (x_{1}-a)+h,
\end{equation*}%
where $\deg _{x_{1}}(h)=0$. This means that $h$ depends only on the
variables $x_{2},\ldots ,x_{n}$. If there exists $a_{2}\in A_{2},\ldots
,a_{n}\in A_{n}$ so that $h(a_{2},\ldots ,a_{n})\neq 0$, then we get $%
f(a,a_{2},\ldots ,a_{n})=h(a_{2},\ldots ,a_{n})\neq 0$, which proves the
assertion. Otherwise $h|_{A_{2}\times \ldots \times A_{n}}\equiv 0$. By the
division algorithm we have 
\begin{equation*}
\limfunc{Supp}(g)\subseteq \{(\alpha _{1}-r,\alpha _{2},\ldots ,\alpha
_{n}):\;(\alpha _{1},\alpha _{2},\ldots ,\alpha _{n})\in \limfunc{Supp}%
(f),\;1\leq r\leq \alpha _{1}\},
\end{equation*}%
and $(\alpha _{1}-1,\alpha _{2},\ldots ,\alpha _{n})\in \limfunc{Supp}(g)$.
Thus the tuple $(\alpha _{1}-1,\alpha _{2},\ldots ,\alpha _{n})$ is maximal
in $\limfunc{Supp}(g)$. By inductive assumption we know that there exist $%
a_{1}\in A_{1}\setminus {\{a\}},a_{2}\in A_{2},\ldots ,a_{n}\in A_{n}$ so
that $g(a_{1},\ldots ,a_{n})\neq 0$. Hence 
\begin{equation*}
f(a_{1},a_{2},\ldots ,a_{n})=(a_{1}-a)\cdot g(a_{1},\ldots ,a_{n})\neq 0,
\end{equation*}%
which proves the assertion of the theorem.
\end{proof}

\section{Coefficient formula}

Let $\mathbb{F}$ be an arbitrary field and let $A_{1},\ldots ,A_{n}$ be any
finite subsets of $\mathbb{F}$. Define the function $N:A_{1}\times \ldots
\times A_{n}\rightarrow \mathbb{F}$ by%
\begin{equation*}
N(a_{1},\ldots ,a_{n})=\prod_{i=1}^{n}\prod_{b\in A_{i}\setminus
\{a_{i}\}}(a_{i}-b).
\end{equation*}%
We may think of the function $N$ as a normalizing factor for the
interpolating function on $A_{1}\times ...\times A_{n}$ defined by 
\begin{equation*}
\chi _{(a_{1},\ldots ,a_{n})}(x_{1},\ldots ,x_{n})=N(a_{1},\ldots
,a_{n})^{-1}\cdot \prod_{i=1}^{n}\prod_{b\in A_{i}\setminus
\{a_{i}\}}(x_{i}-b).
\end{equation*}%
Notice that $\chi _{(a_{1},\ldots ,a_{n})}$ is everywhere zero on $%
A_{1}\times \ldots \times A_{n}$, except at the point $(a_{1},\ldots ,a_{n})$
for which it takes the value of $1$.

We will need the following simple lemma.

\begin{lemma}
\label{L}Let $A$ be any finite subset of the field $\mathbb{F}$, with $%
\left\vert A\right\vert \geq 2$. Then 
\begin{equation*}
\sum_{a\in A}\prod_{b\in A\setminus \{a\}}(b-a)^{-1}=0.
\end{equation*}
\end{lemma}

\begin{proof}
Consider the polynomial%
\begin{equation*}
f(x)=\sum_{a\in A}\prod_{b\in A\setminus \{a\}}\frac{(x-b)}{(a-b)}.
\end{equation*}%
Its degree is at most $\left\vert A\right\vert -1$, and for all $a\in A$ it
takes value of $1$. Hence $f\equiv 1$ and the coefficient of $x^{\left\vert
A\right\vert -1}$ equals $0$. But it is also the same as the the left hand
side of the asserted equality.
\end{proof}

\begin{theorem}
\label{CF}\emph{(Coefficient Formula)} Let $f$ be a polynomial in $\mathbb{F}%
[x_{1},\ldots ,x_{n}]$ and let $f_{\alpha _{1},\ldots ,\alpha _{n}}$ denote
the coefficient of $x_{1}^{\alpha _{1}}\cdots x_{n}^{\alpha _{n}}$ in $f$.
Suppose that $(\alpha _{1},\ldots ,\alpha _{n})$ is maximal in $\limfunc{Supp%
}(f)$. Then for any sets $A_{1},\ldots ,A_{n}$ in $\mathbb{F}$ such that $%
\left\vert A_{i}\right\vert =\alpha _{i}+1$ we have%
\begin{equation}
f_{\alpha _{1},\ldots ,\alpha _{n}}=\sum_{(a_{1},\ldots ,a_{n})\in
A_{1}\times \ldots \times A_{n}}\frac{f(a_{1},\ldots ,a_{n})}{N(a_{1},\ldots
,a_{n})}.  \tag{*}
\end{equation}
\end{theorem}

\begin{proof}
The proof is by induction on the number of elements in the set%
\begin{equation*}
\limfunc{Cone}(f)=\{\beta \in \mathbb{N}^{n}:\text{there\ exists}\;\alpha
\in \limfunc{Supp}(f)\;\text{and}\;\alpha \geq \beta \}.
\end{equation*}%
If $\left\vert \limfunc{Cone}(f)\right\vert =0$ then $f\equiv 0$ and the
theorem is trivial. Otherwise let $(\beta _{1},\ldots ,\beta _{n})$ be a
maximal element of $\limfunc{Cone}(f)$, so it also belongs to $\limfunc{Supp}%
(f)$. If $(\beta _{1},\ldots ,\beta _{n})=(\alpha _{1},\ldots ,\alpha _{n})$%
, then consider the polynomial 
\begin{equation*}
f^{\prime }(x_{1},\ldots ,x_{n})=f(x_{1},\ldots ,x_{n})-f_{\alpha
_{1},\ldots ,\alpha _{n}}\cdot \prod_{i=1}^{n}\prod_{b\in A_{i}\setminus
\{a_{i}\}}(x_{i}-b)
\end{equation*}%
for arbitrary $a_{1}\in A_{1},\ldots ,a_{n}\in A_{n}$. Notice that 
\begin{equation*}
\limfunc{Cone}(f^{\prime })\subset \limfunc{Cone}(f)\setminus \{(\alpha
_{1},\ldots ,\alpha _{n})\},
\end{equation*}%
so from inductive assumption we get the assertion for polynomial $f^{\prime
} $. Since (*) is $\mathbb{F}$-linear, to prove it for $f$ it is enough to
prove it for the polynomial 
\begin{equation*}
h=f_{\alpha _{1},\ldots ,\alpha _{n}}^{-1}\cdot (f-f^{\prime
})=\prod_{i=1}^{n}\prod_{b\in A_{i}\setminus \{a_{i}\}}(x_{i}-b).
\end{equation*}%
But now the equality is obvious: $h(x_{1},\ldots ,x_{n})\neq 0$ only for $%
(a_{1},\ldots ,a_{n})$ and therefore $h(a_{1},\ldots ,a_{n})=N(a_{1},\ldots
,a_{n})$.

If $(\beta _{1},\ldots ,\beta _{n})\neq (\alpha _{1},\ldots ,\alpha _{n})$
then by the assumptions we have $(\beta _{1},\ldots ,\beta _{n})\ngtr
(\alpha _{1},\ldots ,\alpha _{n})$. So there exists $i$ such that $\beta
_{i}<\alpha _{i}$, without loss of generality we can assume that $\beta
_{1}<\alpha _{1}$. Let $B_{1}\subset A_{1}$ be any subset with $\beta _{1}$
elements. So, we have $\left\vert A_{1}\setminus B_{1}\right\vert \geq 2$.
Consider the polynomial 
\begin{equation*}
f^{\prime }(x_{1},\ldots ,x_{n})=f(x_{1},\ldots ,x_{n})-f_{\beta _{1},\ldots
,\beta _{n}}\cdot x_{2}^{\beta _{2}}\cdots x_{n}^{\beta _{n}}\cdot
\prod_{b\in B_{1}}(x_{1}-b).
\end{equation*}%
As before we have that 
\begin{equation*}
\limfunc{Cone}(f^{\prime })\subset \limfunc{Cone}(f)\setminus \{(\beta
_{1},\ldots ,\beta _{n})\},
\end{equation*}%
so, from inductive assumption we get the assertion for polynomial $f^{\prime
}$. It remains to prove it for the polynomial 
\begin{equation*}
h=f_{\beta _{1},\ldots ,\beta _{n}}^{-1}\cdot (f-f^{\prime })=x_{2}^{\beta
_{2}}\cdots x_{n}^{\beta _{n}}\cdot \prod_{b_{1}\in B_{1}}(x_{1}-b_{1}).
\end{equation*}%
Obviously, the right-hand side of equality (*) equals zero. After rewriting
the left-hand side we get 
\begin{eqnarray*}
&&\sum_{a\in A_{1}\times \ldots \times A_{n}}\left\{
\prod_{i=1}^{n}\prod_{b\in A_{i}\setminus \{a_{i}\}}(b-a)\right\} ^{-1}\cdot
x_{2}^{\beta _{2}}\cdots x_{n}^{\beta _{n}}\cdot \prod_{b\in B_{1}}(x_{1}-b)=
\\
&=&\sum_{a_{2}\in A_{2},\ldots ,a_{n}\in A_{n}}\prod_{i=2}^{n}\prod_{b\in
A_{i}\setminus \{a_{i}\}}\left( (b-a_{i})^{-1}\cdot x_{i}^{\beta
_{i}}\right) \cdot \\
&&\cdot \left( \sum_{a\in A_{1}}\prod_{b\in A_{1}\setminus
\{a\}}(b-a)^{-1}\prod_{b\in B_{1}}(x_{1}-b)\right)
\end{eqnarray*}%
The last factor in this product can be simplified to the form 
\begin{equation*}
\sum_{a\in A_{1}\setminus B_{1}}\prod_{b\in (A_{1}\setminus B_{1})\setminus
\{a\}}(b-a)^{-1},
\end{equation*}%
which is zero by the Lemma \ref{L}. The proof is completed.
\end{proof}

Notice that Theorem \ref{CF} implies Theorem \ref{GCN}. Indeed, if $%
f_{\alpha _{1},\ldots ,\alpha _{n}}\neq 0$, then $f$ cannot vanish on every
point of $A_{1}\times \ldots \times A_{n}$. Also if $f_{\alpha _{1},\ldots
,\alpha _{n}}=0$, then either $f$ vanishes on the whole set $A_{1}\times
\ldots \times A_{n}$, or there are at least two points for which $f$ takes a
non-zero value.

\section{Applications}

In this section we give an example of possible application of Theorem \ref%
{GCN}. In some sense it generalizes the idea of lucky labelings of graphs
from \cite{CzerwinskiIPL}. Given a simple graph $G=(V,E)$ and any function $%
c:V\rightarrow \mathbb{N}$, let $S(u)=\sum\nolimits_{v\in N(u)}c(v)$ denote
the sum of labels over the set $N(u)$ of all neighbors of $u$ in $G$. The
function $c$ is called a \emph{lucky labeling} of $G$ if $S(u)\neq S(w)$ for
every pair of adjacent vertices $u$ and $w$.  The main conjecture from \cite%
{CzerwinskiIPL} states that every $k$-colorable graph has a lucky labeling
with values in the set $\{1,2,\ldots ,k\}$. One of the results of \cite%
{CzerwinskiIPL} in this direction asserts that the set of labels $\{1,2,3\}$
is sufficient for every bipartite planar graph $G$. This result is a special
case of the following general theorem.

\begin{theorem}
Let $G$ be a bipartite graph, which has an orientation with outgoing degree
bounded by $k$. Suppose each vertex $v$ is equipped with a non-constant
polynomial $f_{v}\in \mathbb{R}[x]$ of degree at most $l$ and positive
leading coefficient. Then there is a labeling $c:V(G)\rightarrow
\{1,2,\ldots ,kl+1\}$ such that for any two adjacent vertices $u$ and $w$,%
\begin{equation*}
c(u)-\sum_{v\in N(u)}f_{v}(c(v))\neq c(w)-\sum_{v\in N(w)}f_{v}(c(v)).
\end{equation*}
\end{theorem}

\begin{proof}
Assign to each vertex $v\in V(G)$ a variable $x_{v}$. Consider the
polynomial  
\begin{equation*}
h=\prod_{uw\in E(G)}(\sum_{v\in N(u)}f_{v}(x_{v})+x_{w}-\sum_{v\in
N(w)}f_{v}(x_{v})-x_{u})
\end{equation*}%
We want to show that we can choose values for $x_{v}$ from the set $%
\{1,\ldots ,kl+1\}$ so that $h$ is non-zero. Let us fix an orientation of $G$
where outgoing degree is bounded by $k$. For each edge $uw\in E(G)$ oriented 
$u\rightarrow w$ choose the leading monomial in $f_{u}(x_{u})$ from the
factor corresponding to this edge in $h$. The product of this monomials is a
monomial $M$ of $h$ satisfying $\deg _{x_{v}}(M)\leq kl$ (since monomials
from $f_{u}(x_{u})$ are taken at most $k$ times). We claim that the
coefficient of $M$ in $h$ is nonzero. Indeed, each time we take a product of
monomials from factors of $h$ resulting in the monomial $M$, the sign of $M$
is the same (because $G$ is bipartite and leading coefficients of $f_{v}$
are positive). So the copies of $M$ cannot cancel as we are working in the
field $\mathbb{R}$. Finally, maximality of $M$ in $\limfunc{Supp}(h)$ can be
seen easily by collapsing each polynomial $f_{v}(x_{v})$ to $x_{v}$. The
assertion follows from Theorem \ref{GCN}.
\end{proof}

Notice that in the above theorem the labels can be taken from arbitrary
lists of size at least $kl+1$.

Let us conclude the paper with the following remark. Suppose that we want to
use classical Combinatorial Nullstellensatz to the polynomial $%
f(x_{1},\ldots ,x_{n})$ of degree $\sum_{i=1}^{n}\alpha _{i}$ with nonzero
coefficient of $x_{1}^{\alpha _{1}}\cdots x_{n}^{\alpha _{n}}$. If $%
f(x_{1},\ldots ,x_{n})=g(h(x_{1}),x_{2},\ldots ,x_{n})$ with $\deg (h)=k$,
then for arbitrary sets $A_{1},A_{2},\ldots ,A_{n}\subset \mathbb{F}$, with $%
h(a)\neq h(b)$ for all distinct $a,b\in A_{1}$ and of size at least $\alpha
_{1}/k+1,\alpha _{2}+1,\ldots ,\alpha _{n}+1$, $f$ does not vanish on $%
A_{1}\times \ldots \times A_{n}$. So we gain almost $k$ times smaller first
set in comparison with the classical version. It is an immediate consequence
of substitution $x_{1}^{\prime }:=h(x_{1})$ and Theorem \ref{GCN} applied to 
$f^{\prime }(x_{1}^{\prime },x_{2},\ldots ,x_{n})=f(x_{1},\ldots ,x_{n})$.
Analogous corollary is true for more variables being in fact equal to some
polynomials.

\section*{Acknowledgement} 
I would like to thank Jarek Grytczuk for stimulating discussions on the polynomial method in Combinatorics.


\begin{thebibliography}{9}
\bibitem{AlonCPC} N. Alon, Combinatorial Nullstellensatz, Comb. Prob.
Comput. 8 (1999), 7-29.

\bibitem{CzerwinskiIPL} S. Czerwi\'{n}ski, J. Grytczuk, W. \.{Z}elazny,
Lucky labelings of graphs, Inform. Process. Lett. 109 (2009), 1078-1081.

\bibitem{Kouba} O. Kouba, A duality based proof of the Combinatorial
Nullstellensatz, Electron. J. Combin. 16 (2009), Note 9, 3 pp.

\bibitem{Michalek} M. Micha\l ek, A short proof of Combinatorial
Nullstellensatz, Amer. Math. Monthly 117 (2010), no. 9, 821-823.

\bibitem{Schauz} U. Schauz, Algebraically solvable problems: describing
polynomials as equivalent to explicit solutions. Electron. J. Combin. 15
(2008), no. 1, Research Paper 10, 35 pp.
\end{thebibliography}
\end{document}